\documentclass{article}
\usepackage{amsmath}
\usepackage{amsfonts}

\def\ep{\varepsilon}

\def\lip{\hskip0.02cm{\rm Lip}\hskip0.01cm}

\def\dist{\hskip0.02cm{\rm dist}\hskip0.01cm}
\def\supp{\hskip0.02cm{\rm supp}\hskip0.01cm}

\newtheorem{theorem}{Theorem}

\newenvironment{remark}[1]{\medskip\par\noindent{\bf #1.}\rm}
{\medskip\par\noindent}

\newenvironment{proof}[1]{\medskip\par\noindent{\sc #1.\ }}
{~\rule{0.5em}{0.5em}\medskip\par}

\begin{document}

\title{\LARGE Expansion properties of metric spaces not admitting a coarse embedding into a Hilbert space}

\author{M.~I.~Ostrovskii\\
Department of Mathematics and Computer Science\\
St. John's University\\
8000 Utopia Parkway\\
Queens, NY 11439, USA\\
e-mail: {\tt ostrovsm@stjohns.edu}}

\date{\today}
\maketitle
\begin{large}

\noindent{\bf Abstract.} The main purpose of the paper is to find
some expansion properties of locally finite metric spaces which do
not embed coarsely into a Hilbert space. The obtained result is
used to show that infinite locally finite graphs excluding a minor
embed coarsely into a Hilbert space. In an appendix a direct proof
of the latter result is given.
\medskip


A metric space $(M,d_M)$ is called {\it locally finite} if all
balls in it have finitely many elements. We say that $(M,d_M)$ has
{\it bounded geometry} if for each $r>0$ there is $U(r)<\infty$
such that each ball of radius $r$ in $M$ has at most $U(r)$
elements. Let $A$ and $B$ be metric spaces. A mapping $f:A\to B$
is called a {\it coarse embedding} if there exist non-decreasing
functions $\rho_1,\rho_2:[0,\infty)\to[0,\infty)$ such that {\bf
(1)} $\forall x,y\in A~ \rho_1(d_A(x,y))\le
d_B(f(x),f(y))\le\rho_2(d_A(x,y))$; {\bf (2)}
$\lim_{r\to\infty}\rho_1(r)=\infty$. \medskip

We are interested in conditions under which a locally finite
metric space $M$ embeds coarsely into a Hilbert space.  See
\cite{Gro93}, \cite{Roe03}, and \cite{Yu06} for motivation and
background for this problem. Since, as it is well-known (see
e.\,g. \cite[Section 4]{Ost09}), coarse embeddability into a
Hilbert space is equivalent to coarse embeddability into $L_1$, we
consider coarse embeddability into $L_1$.
\medskip

Locally finite metric space which are not coarsely embeddable into
$L_1$ were characterized in \cite{Ost09} and \cite{Tes09}. We
reproduce the characterization as it is stated in \cite{Ost09}.

\begin{theorem}[{\cite[Theorem 2.4]{Ost09}}]\label{T:mu} Let $(M,d_M)$ be a locally
finite metric space which is not coarsely embeddable into $L_1$.
Then there exists a constant $D$, depending on $M$ only, such that
for each $n\in\mathbb{N}$ there exists a finite set $M_n\subset M$
and a probability measure $\mu_n$ on $M_n\times M_n$ such that
\begin{itemize}
\item $d_M(u,v)\ge n$ for each $(u,v)\in\supp\mu_n$.

\item For each Lipschitz function $f:M\to L_1$ we have
\begin{equation}\label{E:OTE}\int_{M_n\times M_n}||f(u)-f(v)||_{L_1}d\mu_n(u,v)\le D\lip(f).\end{equation}
\end{itemize}
\end{theorem}

Our first purpose is to find some expansion properties of sets
$M_n$.\medskip

Let $s$ be a positive integer. We consider graphs $G(n,s)=(M_n,
E(M_n,s))$, where the edge set $E(M_n,s)$ is obtained by joining
those pairs of vertices of $M_n$ which are at distance $\le s$.
The graphs $\{G(n,s)\}_{n=1}^\infty$ have uniformly bounded
degrees if the metric space $M$ has bounded geometry.
\medskip

\noindent{\bf Observation:} Each vertex cut of $G(n,s)$ separates
it into pieces with $d_M$-distance between then at least $s$.
\medskip

If we would prove in the bounded geometry case that the condition
\begin{itemize}
\item[(*)] For some $s\in\mathbb{N}$ there is a number $h_s>0$ and
subgraphs $H_n$ of $G(n,s)$ of indefinitely growing sizes (as
$n\to\infty$) such that the expansion constants of $\{H_n\}$ are
uniformly bounded from below by $h_s$
\end{itemize}
is satisfied, it would solve the well-known problem (see
\cite{GK04}, \cite{Ost09}, \cite{Tes09}): whether each metric
space with bounded geometry which does not embed coarsely into a
Hilbert space contains weak expanders? For spaces with bounded
geometry weak expanders are defined as Lipschitz images $f_m(X_m)$
of (vertex sets) of a family of expanders with uniformly bounded
Lipschitz constants of $\{f_m\}_{m=1}^\infty$ and without
dominating pre-images in the sense that
$\displaystyle{\lim_{m\to\infty}\max_{z\in
f_m(X_m)}}(|f_m^{-1}(z)|/|X_m|)=0$.

\begin{remark}{Remark} When we consider a connected graph
as a metric space, we identify the graph with its vertex set
endowed with the standard graph distance.
\end{remark}

The well-known proof of non-embeddability of expanders (see
\cite{Gro00}, \cite{Mat97}, \cite[Section 11.3]{Roe03}) shows that
a metric space with bounded geometry containing weak expanders
does not embed coarsely into a Hilbert space.)
\medskip

In this paper we prove only the following weaker expansion
property of the graphs $G(n,s)$. We introduce the measure $\nu_n$
on $M_n$ by $\nu_n(A)=\mu_n(A\times M_n)$. Let $F$ be an induced
subgraph of $G(n,s)$. We denote the vertex boundary of a set $A$
of vertices in $F$ by $\delta_F A$.

\begin{theorem}\label{T:mu-exp} Let $s$ and $n$ be such that
$2n>s>8D$. Let $\varphi(D,s)=\frac{s}{4D}-2$. Then $G(n,s)$
contains an induced subgraph $F$ with $d_M$-diameter $\ge n-\frac
s2$, such that each subset $A\subset F$ of $d_M$-diameter
$<n-\frac s2$ satisfies the condition: $\nu_n(\delta_F A)>
\varphi(D,s)\nu_n(A)$.
\end{theorem}

\begin{proof}{Proof of Theorem \ref{T:mu-exp}} Suppose that for some $n,s\in\mathbb{N}$
satisfying $2n>s>8D$ there is no such subgraph in $G(n,s)$. Then
for each induced subgraph $F$ in $G(n,s)$ of $d_M$-diameter $\ge
n-\frac s2$ we can find a subset $A\subset F$ of $d_M$-diameter
$<n-\frac s2$ such that $\nu_n(\delta_F
A)\le\varphi(D,s)\nu_n(A)$. We start with $F_1=G(n,s)$ (the
definitions of $M_n$ and $\mu_n$ imply that the $d_M$-diameter of
$M_n$ is $\ge n$), find a subset $A_1\subset F_1$ of
$d_M$-diameter $<n-\frac{s}2$ such that $\nu_n(\delta_{F_1}
A_1)\le\varphi(D,s)\nu_n(A_1)$, and remove $A_1\cup\delta_{F_1}
A_1$ from $G(n,s)$. If the obtained graph $F_2$ still has
$d_M$-diameter $\ge n-\frac s2$, we find a subset $A_2$ in it such
that $\nu_n(\delta_{F_2} A_2)\le \varphi(D,s)\nu_n(A_2)$. We
remove the subset $A_2\cup\delta_{F_2} A_2$ from $F_2$. We
continue in an obvious way till we get a set of $d_M$-diameter
$<n-\frac s2$ (this should eventually happen since $M_n$ is
finite). We denote this set $A_p$, where $p$ is the number of
steps in the process.\medskip

\begin{remark}{Remark} This exhaustion process is similar to the
one used in \cite{LS93}.
\end{remark}

Observe that each of the sets $A_i$ has diameter $<n-\frac s2$,
and that the $d_M$-distance between any $A_i$ and $A_j$ $(i\ne j)$
is at least $s$ (see the observation above).
\medskip

We introduce a family of  $1$-Lipschitz functions $f_\theta$ on
$M$, where $\theta=\{\theta_i\}_{i=1}^p\in\Theta=\{-1,1\}^p$ by
the formula:
$$f_\theta(x)=\begin{cases}\theta_j\displaystyle{\left(\frac {s}2-\dist(x,A_j)\right)} &\hbox{ if }\dist(x,A_j)<\frac s2\\
0 & \hbox{ if }\hbox{dist}(x,\cup_{i=1}^p A_i)\ge \frac
s2.\end{cases}$$ The function is well-defined since the inequality
$\dist(x,A_j)<\frac s2$ cannot be satisfied for more than one
value of $j$. Straightforward verification shows that this
function is $1$-Lipschitz.\medskip

We endow $\Theta=\{-1,1\}^p$ with the natural probability measure
$\mathcal{P}$ and introduce for each $x\in M$ a function $F_x\in
L_1(\Theta,\mathcal{P})$ given by $F_x(\theta)=f_\theta(x)$. It is
clear that the mapping  $x\mapsto F_x$ is $1$-Lipschitz.
\medskip

Applying inequality \eqref{E:OTE} to this mapping we get
\[\begin{split}D&\ge\int_{M_n\times M_n}||F_x(\theta)-F_y(\theta)||_{L_1(\Theta,\mathcal{P})}d\mu_n(x,y)\ge
\int_{M_n\times M_n}\int_\Theta
|f_\theta(x)-f_\theta(y)|d\mathcal{P}(\theta)d\mu_n(x,y)\\&\ge
\int_{M_n\times
M_n}\int_{\Psi(x,y)}|f_\theta(x)|d\mathcal{P}(\theta)d\mu_n(x,y),\end{split}\]
where $\Psi(x,y)$ is the subset of $\Theta$ for which
$f_\theta(x)$ and $f_\theta(y)$ have different signs (we mean that
signs have values in $\{-1,0,1\}$). Observe that the value of
$|f_\theta(x)|$ does not depend on $\theta$. We get

\[\begin{split}
\int_{M_n\times
M_n}\int_{\Psi(x,y)}|f_\theta(x)|d\mathcal{P}(\theta)d\mu_n(x,y)&\ge
\int_{\left(\cup_{i=1}^pA_i\right)\times
M_n}|f_\theta(x)|\int_{\Psi(x,y)}d\mathcal{P}(\theta)d\mu_n(x,y).
\end{split}\]

Now we observe that for $x\in A_j$ and $y$ satisfying
$(x,y)\in\supp\mu_n$ we have $d_M(x,y)\ge n$ and therefore
$d_M(y,A_j)\ge\frac s2$ (recall that the diameter of $A_j$ is $<
n-\frac s2$). Hence $\mathcal{P}(\Psi(x,y))\ge\frac12$ for each
pair $(x,y)$ from $\supp\mu_n$. We get

\[\begin{split}
\int_{\left(\cup_{i=1}^pA_i\right)\times
M_n}|f_\theta(x)|\int_{\Psi(x,y)}d\mathcal{P}(\theta)d\mu_n(x,y)
&\ge \int_{\left(\cup_{i=1}^pA_i\right)\times M_n}\frac
s2\cdot\frac12d\mu_n(x,y)
\\
&= \frac s4\, \nu_n\left(\cup_i A_i\right). \end{split}\]

\begin{remark}{Remark} The idea of ``random'' signing of functions
in a similar situation was used in \cite{Rao99}.
\end{remark}

Recalling the beginning of this chain of inequalities, we get
\begin{equation}\label{E:DvsNu}D\ge\frac s4\, \nu_n\left(\cup_i A_i\right).
\end{equation}

Observe that $\nu_n(\cup_i A_i)+\nu_n(\cup_i \delta_{F_i}A_i)=1$
and $\nu_n(\cup_i \delta_{F_i} A_i)\le\varphi(D,s)\nu_n(\cup_i
A_i)$. Therefore
\begin{equation}\label{E:DvsPhi} (1+\varphi(D,s))\nu(\cup_i A_i)\ge 1\end{equation}
Combining \eqref{E:DvsNu} and \eqref{E:DvsPhi} we get
$$D\ge\frac s{4(1+\varphi(D,s))},$$
or $\varphi(D,s)\ge \frac s{4D}-1$, a contradiction.
\end{proof}

Now we combine Theorem \ref{T:mu-exp} with some results and
technique from \cite{KPR93} (some of the estimates from
\cite{KPR93} were improved in \cite{FT03} but we do not use this
improvement).

\begin{theorem}\label{T:MinorExcl} Let $r\in\mathbb{N}$ and $G$ be a locally finite connected graph which does not have $K_r$-minors, let $d_G$ be the graph distance on
$G$. Then $(G,d_G)$ embeds coarsely into $L_1$.
\end{theorem}

\begin{proof}{Proof} Assume the contrary. We apply Theorem
\ref{T:mu} to $G$ and denote by $D$, $M_n$, and $\mu_n$ the
corresponding constant (depending only on $G$), finite sets, and
probability measures. Let $\nu_n$ be measures introduced in
Theorem \ref{T:mu-exp}. According to Theorem \ref{T:mu-exp} for
each $2n>s>8D$ there is an induced subgraph $F=F(n,s)$ in $G(n,s)$
such that the condition of Theorem \ref{T:mu-exp} is satisfied.
The condition $\nu_n(\delta_F A)>\varphi(D,s)\nu_n(A)$ implies
that $\nu_n(F)>0$.
\medskip

Now we use a modified construction from \cite[Section 4]{KPR93}.
Let $t,s\in\mathbb{N}$ (we shall specify our choice of these
numbers later). Let $\Delta=t+2s$. We pick a vertex $x_1\in G$,
$\alpha\in\{0,1,2,\dots,\Delta-1\}$, and let
$$D_1= \{v\in G:~
{(d_G(v,x_1)-\alpha)\pmod{\Delta}}\in\{1,2,\dots,2s\}\}$$ (that
is, $D_1$ consists of infinitely many `annuluses' of width $2s$
each, with distances $t$ between them). We choose $\alpha$ in such
a way that $\nu_n(D_1\cap F)$ is the minimal possible. Using
averaging argument we get that $\alpha$ can be chosen in such a
way that $\nu_n(D_1\cap F)\le\left(\frac
{2s}{2s+t}\right)\nu_n(F)$.
\medskip

We delete $D_1$ from $G$. The second round of deletions is: we
repeat the same procedure for each of the components of the
obtained graph endowed with its own graph distance. Each time we
choose the corresponding $\alpha$ (the level of cut) in such a way
$\nu_n(D\cap F)\le\left(\frac {2s}{2s+t}\right)\nu_n(F\cap X)$,
where $X$ is the component under consideration and $D$ is the set
of vertices deleted this time.
\medskip

We do $r$ rounds of deletions. Let $\{G_i\}$ be the components of
the remaining graph. The argument of \cite[Theorem 4.2]{KPR93}
shows that the $d_G$-diameter of each of $G_i$ does not exceed
$(r-1)(4(r+1)t+1)$ (where $r$ is from the statement of the
theorem). It is also easy to see that
\begin{equation}
\nu_n\left(F\cap \left(\cup_iG_i\right)\right)\ge\left(\frac
t{2s+t}\right)^r\nu_n(F).
\end{equation}

Now we impose additional conditions on $s$, $t$, and $n$ (the
condition $2n>s>8D$ was imposed in Theorem \ref{T:mu-exp}) The
conditions are
\begin{equation}\label{E:vv}
(\varphi(D,s)+1)\left(\frac t{2s+t}\right)^r>1
\end{equation}
\begin{equation}\label{E:vvv}
(r-1)(4(r+1)t+1)< n-\frac s2.
\end{equation}

These conditions can be satisfied. In fact, we choose $s>8D$
first. Then we choose $t$ such that \eqref{E:vv} is satisfied, and
then $n$ such that \eqref{E:vvv} is satisfied.\medskip

Let $R_i=F\cap G_i$. Our choice of parameters implies that the
$d_G$-diameter of $R_i$ is $<n-\frac s2$. Therefore
$\nu_n(\delta_FR_i)>\varphi(D,s)\nu_n(R_i)$. Since $\{\delta_F
R_i\}$ are disjoint (this was the reason why we deleted
`annuluses' of width $2s$), we get \[\begin{split}\nu_n(F)&\ge
\nu_n(\cup_i\delta_F R_i)+\nu_n(\cup_i R_i)>
(\varphi(D,s)+1)\nu_n(\cup_iR_i)\\
&\ge (\varphi(D,s)+1)\left(\frac
t{2s+t}\right)^r\nu_n(F).\end{split}\] We get a contradiction with
\eqref{E:vv}.
\end{proof}

\noindent{\bf Appendix: Coarse embeddability of graphs with
excluded minors. Second proof}
\medskip

The purpose of this appendix is to show that coarse embeddability
of graphs excluding $K_r$ as a minor can be proved using the
techniques from \cite{KPR93} and \cite{Rao99}  (see also
\cite{FT03}), without using Theorems \ref{T:mu} and
\ref{T:mu-exp}.

\begin{proof}{Second proof of Theorem \ref{T:MinorExcl}} For $\Delta\in\mathbb{N}$ by $[\Delta]$ we denote the
set $\{1,\dots,\Delta\}$. For each $\Delta\in\mathbb{N}$ we
consider the probability space
\begin{equation}\label{E:prod}\Omega_\Delta=\Lambda_{\Delta}\times\Theta,\end{equation} where
$$\Lambda_{\Delta}=[\Delta]^r\hbox{ and }\Theta=\{-1,1\}^{\mathbb{N}}.$$

For each point $\omega\in\Omega_\Delta$ we define a function
$f_{\Delta,\omega}:X\to\mathbb{R}$ in the following way.
\medskip

We assume that elements of $X$ are enumerated, so
$X=\{x_k:~k\in\mathbb{N}\}$. Let $$(\{r_j\}_{j=1}^r,
\{\theta_j\}_{j=1}^\infty)\in\Omega_\Delta$$
\medskip

We denote by $D_1$ the set of all vertices $v$ in $X$ with
$d(v,x_1)=r_1\pmod{\Delta}$.
\medskip

We delete the set $D_1$ from $X$. We label connected components of
the obtained graph by the numbers of the least subscripts of
vertices contained in them. For the component where $x_j$ is the
vertex with the least subscript, we do the same procedure as above
(with the respect to the graph distance defined by the subgraph)
with $d(v,x_j)=r_2\pmod{\Delta}$. So the number $r_2$ is used for
all of the components of this level.
\medskip

We denote the set of all obtained vertices by $D_2$ and delete it
from the graph. We repeat the procedure $r$ times. Let
$\{X_i\}_{i=1}^\infty$ be components of the obtained graph.
\medskip

We define the function $f_{\Delta,\omega}(u)$ corresponding to
$\omega=(\{r_j\}_{j=1}^r, \{\theta_j\}_{j=1}^\infty)$ by
$$f_{\Delta,\omega}(u)=\theta_k\dist\left(u,\cup_{i=1}^n D_i\right),$$
where $k$ is the least subscript of a point $x_k$ belonging to the
same component of $X\backslash(\cup_{i=1}^r D_i)$ as $u$. An
obvious and very important property of $f_{\Delta,\omega}$ is that
it is a real-valued $1$-Lipschitz function.
\medskip

One of the main results of \cite{KPR93} (Theorem 4.2) (see also
\cite{FT03}) implies that the diameters of the components $X_i$
are $<(r-1)(4(r+1)\Delta+1)=:d_{\Delta,r}$.
\medskip

Now, for each vertex $u$ in $X$ we introduce a function
$F_{\Delta,u}(\omega)$ in $L_1(\Omega_\Delta)$ given by
$$F_{\Delta,u}(\omega)=f_{\Delta,\omega}(u)$$
It is easy to see that $|F_{\Delta,u}(\omega)|\le\Delta/2$ for all
$u$ and $\omega$. The function $F_{\Delta, u}(\omega)$ is
measurable because all subsets of $\Omega_\Delta$ are measurable.
(It is worth mentioning that for each $u$ the value of the
function at $\omega$ depends only on finitely many values of
$\theta_i$. In fact, for a fixed $u$ the value of
$f_{\Delta,\omega}(u)$ can depend only on those $\theta_k$ for
which $x_k$ is in the same component $X_i$ as $u$. But for such
$x_k$ we have $d(u,x_k)\le (r-1)(4(r+1)\Delta+1)$. Since $X$ is
locally finite, there are only finitely many $x_k$ satisfying this
condition.)
\medskip

The following inequality is a very important property of the
functions $F_{\Delta,u}$:
\begin{equation}\label{E:bel}\int_{\Omega_\Delta}|F_{\Delta,u}(\omega)|d\omega\ge
\ep_r\Delta,\end{equation} where $\ep_r$ depends on $r$ only (see
\cite[Lemma 3]{Rao99}, the dependence obtained in this way is of
the form $\delta^r$, where $0<\delta<1$). Furthermore, if we write
$\omega=(\lambda,\theta)$ according to \eqref{E:prod}, we have
\begin{equation}\label{E:bel2}\int_{\Omega_\Delta}|F_{\Delta,w}(\lambda,\theta)|d\lambda\ge
\ep_r\Delta~\forall \theta\in\Theta.\end{equation}

If $d(u,v)\ge d_{\Delta,r}$, then $u$ and $v$ are in different
pieces of the decomposition no matter how
$\lambda=\{r_j\}_{j=1}^r$ is chosen. Therefore, with probability
$\frac12$, the signs of $f_{\Delta,\omega}(u)$ and
$f_{\Delta,\omega}(v)$ are different, Let
$\Psi(\lambda)\subset\Theta$ be the subset for which the signs
$f_{\Delta,\lambda,\theta}(u)$ and $f_{\Delta,\lambda,\theta}(v)$
are different. Then
\begin{equation}\label{E:pad}\begin{split}||F_{\Delta,u}-F_{\Delta,v}||_{L_1(\Omega_\Delta)}&=
\int_{\Lambda_\Delta}\int_{\Theta}|F_{\Delta,u}(\lambda,\theta)-F_{\Delta,v}(\lambda,\theta)|
d\theta d\lambda\\
&\ge\int_{\Lambda_\Delta}\int_{\Psi(\lambda)}|F_{\Delta,u}(\lambda,\theta)-F_{\Delta,v}(\lambda,\theta)|d\theta
d\lambda\\&=\int_{\Lambda_\Delta}\int_{\Psi(\lambda)}(|F_{\Delta,u}(\lambda,\theta)|+|F_{\Delta,v}(\lambda,\theta)|)d\theta
d\lambda\\&(\hbox{observe that the integrand does not depend on }
\theta)
\\&=\frac12\int_{\Lambda_\Delta}(|F_{\Delta,u}(\lambda,\theta)|+|F_{\Delta,v}(\lambda,\theta)|)d\lambda
\\&\ge\ep_r\Delta.\end{split}\end{equation}

We apply this construction with $\Delta=2,4\dots,2^i,\dots$. Let
$\Omega=\cup_{i=1}^\infty\Omega_{2^i}$ be the disjoint union of
the measure spaces $\Omega_{2^i}$. Let $O$ be one of the vertices
of $X$. We introduce an embedding $\varphi:X\to L_1(\Omega)$ by
\[\varphi(v)|_{\Omega_{2^i}}=\left(\frac23\right)^i\left(F_{2^i,v}(\omega)-F_{2^i,O}(\omega)\right).\]

To complete the proof of the theorem it remains to show that
$\varphi$ is a well-defined mapping and that it is a coarse
embedding.
\medskip

Since $f_{\Delta,\omega}(u)$ are $1$-Lipschitz (as functions of
$u$) real-valued functions, the mappings
$\varphi_i(v):=F_{2^i,v}\in L_1(\Omega_{2^i})$ are also
$1$-Lipschitz. Therefore
$||\varphi_i(v)-\varphi_i(O)||_{L_1(\Omega_{2^i})}\le d(O,v)$ and
$\varphi(v)\in L_1(\Omega)$.
\medskip

To show that $\varphi$ is a coarse embedding it suffices to
establish the following two inequalities:
\begin{equation}\label{E:3-Lip}
||\varphi(u)-\varphi(v)||_{L_1(\Omega)}\le 3d(u,v),
\end{equation}

\begin{equation}\label{E:below}
d(u,v)\ge d_{2^i,n}\Rightarrow
||\varphi(u)-\varphi(v)||_{L_1(\Omega)}
\ge\left(\frac43\right)^i\ep_r.
\end{equation}

The inequality \eqref{E:3-Lip} is an immediate consequence of the
fact that $\varphi_i$ are $1$-Lipschitz:
\[||\varphi(u)-\varphi(v)||_{L_1(\Omega)}=\sum_{i=0}^\infty\left(\frac23\right)^i||\varphi_i(u)-
\varphi_i(v)||_{L_1(\Omega_{2^i})} \le
d(u,v)\sum_{i=0}^\infty\left(\frac23\right)^i=3d(u,v).\]

If $d(u,v)\ge d_{2^i,n}$, we apply the inequality \eqref{E:pad}
and get
$$||\varphi(u)-\varphi(v)||_{L_1(\Omega)}\ge\left(\frac23\right)^i||\varphi_i(u)-\varphi_i(v)||_{L_1(\Omega_{2^i})}
\ge\left(\frac23\right)^i2^i\ep_r= \left(\frac43\right)^i\ep_r.$$
\end{proof}

\end{large}

\begin{small}

\end{small}

\end{document}